\documentclass[11pt,reqno]{amsart}
\usepackage{color}
\usepackage[english]{babel}
\usepackage{amsfonts}
\usepackage{amsmath}
\usepackage{amsthm}
\usepackage{amssymb}
\usepackage[misc]{ifsym}
\usepackage{bbm}
\usepackage{mathrsfs}
\usepackage{esint}
\usepackage{faktor}
\usepackage[utf8]{inputenc}
\usepackage{afterpage}
\usepackage[left=2.9cm,right=2.9cm,top=2.8cm,bottom=2.8cm]{geometry}
\usepackage{graphicx}
\usepackage{enumitem}
\usepackage[dvipsnames]{xcolor}
\usepackage[colorlinks=true,urlcolor=blue,citecolor=blue,linkcolor=blue,linktocpage,pdfpagelabels,bookmarksnumbered,bookmarksopen]{hyperref}

\newcommand{\doi}[1]{\url{https://doi.org/#1}}
\newcommand{\N}{\mathbb{N}}

\newcommand{\R}{\mathbb{R}}
\newcommand{\D}{\mathcal{D}}

\renewcommand{\L}{\mathscr{L}}
\newcommand{\Div}{\mathrm{div} \, }
\newcommand{\Jac}{\mathrm{Jac} \, }
\newcommand{\dx}{\, {\rm d} x}

\newcommand{\dt}{\, {\rm d} t}

\newcommand{\dzeta}{\, {\rm d} \zeta}
\newcommand{\eps}{\varepsilon}
\newcommand{\loc}{{\rm loc}}

\newcommand{\Lip}{{\rm Lip}}

\renewcommand{\phi}{\varphi}

\newtheorem{lemma}{Lemma}
\newtheorem{thm}[lemma]{Theorem}
\newtheorem{prop}[lemma]{Proposition}
\newtheorem{cor}[lemma]{Corollary}

\theoremstyle{definition}

\newtheorem{rmk}[lemma]{Remark}
\newtheorem{ex}[lemma]{Example}

\DeclareMathOperator*{\esssup}{ess \, sup}
\DeclareMathOperator*{\essinf}{ess \, inf}
\DeclareMathOperator*{\supp}{supp}
\DeclareMathOperator*{\diam}{diam}

\begin{document}
\title[On singular $p$-Laplacian problems with discontinuous convection terms]{On singular $p$-Laplacian problems \\ with discontinuous convection terms}

\author[U. Guarnotta]{Umberto Guarnotta}
\address[U. Guarnotta]{Dipartimento di Matematica e Informatica, Università degli Studi di Catania, Viale A. Doria 6,
95125 Catania, Italy}
\email{umberto.guarnotta@unict.it}

\author[S.A. Marano]{Salvatore A. Marano}
\address[S.A. Marano]{Dipartimento di Matematica e Informatica, Universit\`a degli Studi di Catania, Viale A. Doria 6, 95125 Catania, Italy}
\email{samarano@unict.it}

\begin{abstract}
The existence of positive strong solutions to a homogeneous Dirichlet $p$-Laplacian problem, with reaction sum of a both singular at zero and highly discontinuous nonlinearity and of a discontinuous convection term, is established. Locality arguments, based on suitable measure theoretical results (see Section 3), are employed.
\end{abstract}

\maketitle

{
\let\thefootnote\relax
\footnote{\Letter\quad Corresponding author: Umberto Guarnotta (umberto.guarnotta@unict.it).}
}
\setcounter{footnote}{0}

\section{Introduction}
This paper treats the existence of solutions to the following problem:
\begin{equation}\label{prob}
\tag{P}
\left\{
\begin{alignedat}{2}
-\Delta_p u & =f(u)+g(\nabla u)\quad && \mbox{in}\;\;\Omega,\\
u & >0\quad && \mbox{in}\;\;\Omega,\\
u & =0\quad && \mbox{on}\;\;\partial\Omega,
\end{alignedat}
\right.
\end{equation}
where $\Omega\subseteq\R^N$, $N\geq 2$, is a bounded domain with a $C^2$ boundary $\partial\Omega$, $1<p<N$,  while $f:\R^+\to\R^+_0$ and $g:\R^N\to\R^+_0$ are Borel-measurable functions satisfying the hypotheses below. Denote by $\D_f\subseteq\R^+$ and $\D_g\subseteq\R^N$ the sets of discontinuity points of $f$ and $g$, respectively. Moreover, set
$$\underline{f}(s):=\lim_{\delta\to 0^+}\essinf_{|t-s|<\delta} f(t),\quad
\overline{f}(s):=\lim_{\delta\to 0^+}\esssup_{|t-s|<\delta} f(t)\quad\forall\, s\in\R^+,$$
$$ \underline{g}(\xi):=\lim_{\delta\to 0^+}\essinf_{|\eta-\xi|<\delta} g(\eta),\quad \overline{g}(\xi):=\lim_{\delta\to 0^+}\esssup_{|\eta-\xi|<\delta} g(\eta)\quad\forall\, \xi\in\R^N.$$
We will posit the following conditions.
\begin{enumerate}[label={$({\rm H}_f)$(\roman*)},ref={$({\rm H}_f)${\rm(\roman*)}}]
\item \label{fbound}
$f\in L^\infty_\loc(\R^+)$.
\item \label{fsing}
There exists $\gamma\in(0,1)$ such that $\displaystyle{\limsup_{s\to0^+}s^\gamma f(s)<+\infty}$.
\item \label{fdisc}
$\D_f$ is $\L^1$-negligible, namely $|\D_f|_1=0$.
\item \label{fzeros}
If $\underline{f}(s)=0$ for some $s\in\R^+$, then $f(s)=0$.
\end{enumerate}
\begin{enumerate}[label={$({\rm H}_g)$(\roman*)},ref={$({\rm H}_g)${\rm(\roman*)}}]
\item \label{gbound}
$g\in L^\infty_\loc(\R^N)$.
\item \label{gdisc}
All projections $(\D_g)_i$ of $\D_g$ on the coordinate axes are $\L^1$-negligible, i.e.,
$$|(\D_g)_i|_1=0,\quad i=1,2,\ldots,N.$$
\item \label{gzeros}
If $\underline{g}(\xi)=0$ for some $\xi\in\R^N$, then $g(\xi)=0$.
\end{enumerate}

Let $\lambda_1$ be the first eigenvalue of $(-\Delta_p,W^{1,p}_0(\Omega))$. We will also assume that
\begin{equation}\label{hypatzero}
\tag{${\rm C}_0$}
\mbox{either}\quad\liminf_{s\to 0^+}\frac{f(s)}{s^{p-1}}>\lambda_1\quad
\mbox{or}\quad\liminf_{|\xi|\to 0} g(\xi) > 0,
\end{equation}
as well as
\begin{equation}\label{hypatinfty}
\tag{${\rm C}_\infty$}
\limsup_{s\to+\infty}\frac{f(s)}{s^{p-1}}+\lambda_1^{\frac{p-1}{p}}\limsup_{|\xi|\to\infty}\frac{g(\xi)}{|\xi|^{p-1}}<\lambda_1.
\end{equation}
Hypothesis \eqref{hypatzero} regulates the behavior of $f$ and $g$ at zero while \eqref{hypatinfty} pertains their growth rates at infinity.

%\begin{rmk}
%\label{nojumping}
Surprisingly enough, the presence of convection terms facilitates the existence of solutions to coercive problems. In fact, concerning the non-singular and variational (i.e., $g\equiv0$) framework, a standard case where \eqref{prob} admits solutions is that of `jumping' nonlinearities, namely
$$\liminf_{s\to 0^+}\frac{f(s)}{s^{p-1}}>\lambda_1>
\limsup_{s\to+\infty}\frac{f(s)}{s^{p-1}}.$$
This condition looks quite natural, because if
\begin{equation}\label{nonjumpf}
\frac{f(s)}{s^{p-1}}<\lambda_1\quad\forall\, s>0
\end{equation}
then the existence of a non-trivial solution $\tilde{u}\in W^{1,p}_0(\Omega)$ would imply
$$\|\nabla\tilde{u}\|_p^p=\int_\Omega f(\tilde{u})\tilde{u}\dx<
\lambda_1 \|\tilde{u}\|_p^p\, ,$$
against the variational characterization of $\lambda_1$. Nevertheless, reactions $f$ fulfilling \eqref{nonjumpf} are allowed in our result, provided $\displaystyle{\liminf_{|\xi|\to 0} g(\xi)>0}$; cf. \eqref{hypatzero}.\\
From a technical point of view, assumptions \eqref{hypatzero}--\eqref{hypatinfty} lead to construct a sub-solution that depends on both $f$ and $g$, not merely on $f$ as in, e.g., \cite{LMZ,GMM,GG,BG} (see also \cite{GM0,GMMou} for systems). Accordingly, the so-called freezing technique (namely, fix the gradient term as $g(\nabla v)$ for some function $v$, solve the $v$-dependent problem, and then `unfreeze' $v$ via fixed-point arguments; cf. for instance \cite{GMM}) is no more usable, whence standard variational methods (that were applied on the frozen problems) have to be replaced by more general monotonicity techniques.
%\end{rmk}

Our main result reads as follows.
\begin{thm}\label{mainthm}
Let $({\rm H}_f)$, $({\rm H}_g)$, \eqref{hypatzero}, and \eqref{hypatinfty} be satisfied. Then \eqref{prob} admits a strong solution $u\in C^{1,\alpha}_0(\overline{\Omega})$ for some $\alpha\in(0,1)$.
\end{thm}
Recall that $u\in W^{1,p}_0(\Omega)$ is a strong solution of \eqref{prob} provided $u>0$ in $\Omega$, 
$$|\nabla u|^{p-2}\nabla u\in W^{1,1}_\loc(\Omega;\R^N)$$
(whence $\Delta_p u\in L^1_\loc(\Omega)$), and $-\Delta_p(x)=f(u(x))+g(\nabla u(x))$ for almost all $x\in\Omega$.

%\begin{rmk}
If $g\equiv 0$ then the above result becomes \cite[Theorem 1.3]{GM}; see \cite{GM} and its bibliography for further information concerning either singular or discontinuous elliptic equations. Otherwise, as far as we know, Theorem \ref{mainthm} represents the first contribution on both singular and convective problems with discontinuous nonlinearities. 
%\end{rmk}

Let us shortly outline the main steps of the proof.
\begin{itemize}
\item A sub-solution $\underline{u}\in C^{1,\alpha}_0(\overline{\Omega})$ to \eqref{prob} is constructed via \eqref{hypatzero}.
\item We truncate and regularize both $f$ and $g$, but in a precise way (cf. Remark \ref{ordertruncreg}). Let $F_\eps$ and $G_\eps$, $0<\eps<1$, be the associated superposition operators.
\item A standard result on pseudo-monotone operators yields a solution $u_\eps\in W^{1,p}_0(\Omega)$ to the the equation $-\Delta_p u=F_\eps(u)+G_\eps(u)$ in $W^{-1,p'}(\Omega)$.
\item By the weak comparison principle one has $u_\eps\geq \underline{u}$.
\item Using results from nonlinear regularity theory ensures that $u_\eps\in C^{1,\beta}_0(\overline{\Omega})$.
\item We let $\eps\to0^+$ to get a strong solution $u\in C^{1,\alpha}_0(\overline{\Omega})$ of \eqref{prob} through locality arguments based on suitable measure theoretical results (see Section \ref{measteorres}).
\end{itemize}
Recall that if $X,Y\subseteq L^1_\loc(\Omega)$ are two Banach spaces and $A$ is a map from $X$ to the space of distributions over $\Omega$, then $A$ is called strongly local of $X$ into $Y$ provided every distributional solution $u\in X$ of $A(u)=v$, with $v\in Y$, satisfies
$$ A(u) = 0 \quad\mbox{a.e.\,in}\;\; u^{-1}(c) \quad \mbox{for all}\;\; c\in\R.$$
Some meaningful specific cases are pointed out in \cite[Example 1.5]{GM}. As a sample, by Proposition 2.4 of \cite{GM}, the $p$-Lapacian turns out strongly local from $W^{1,p}(\Omega)$ into $L^2_\loc(\Omega)$.

%Throughout the paper, $c,C,C^*,\tilde{C}$, etc. denote positive constants, whoso dependencies are specified when necessary.
%
%
%
\section{Basic definitions and auxiliary results}\label{prel}
Given two topological spaces $X$ and $Y$, the symbol $X\hookrightarrow Y$ means that $X$ continuously embeds in $Y$.

Let $X$ be a real Banach space with topological dual $X^*$ and duality brackets $\langle\cdot,\cdot\rangle$. A function $A:X\to X^*$ is called:
\begin{itemize}
\item \emph{pseudo-monotone} when
\begin{equation*}
x_n\rightharpoonup x\;\;\mbox{in $X$,}\;\;\limsup_{n\to+\infty}\langle A(x_n),x_n-x\rangle \le 0\implies \liminf_{n\to+\infty}\langle A(x_n),x_n-z\rangle
\geq\langle A(x),x-z\rangle\;\;\forall\, z\in X.   
\end{equation*}
\item \emph{of type $(\mathrm{S})_+$} provided
\begin{equation*}
x_n\rightharpoonup x\;\;\mbox{in $X$,}\;\;\limsup_{n\to+\infty}\langle A(x_n),x_n-x\rangle \le 0\implies x_n\to x\;\;\mbox{in $X$.}   
\end{equation*}
\end{itemize}
For $A,B:X\to X^*$, we know \cite[Section 6]{F} that:
\begin{itemize}
\item[$({\rm P}_1)$] $A,B$ pseudo-monotone $\implies$ $A+B$ pseudo-monotone.
\item[$({\rm P}_2)$] $X$ reflexive, $A$ continuous and of type $(\mathrm{S})_+$ $\implies$ $A$ pseudo-monotone.
\end{itemize}
Moreover (cf. \cite[Lemma 2.2]{GG}), 
\begin{itemize}
\item[$({\rm P}_3)$] $A$ fulfills condition $(\mathrm{S})_+$, $B$ compact $\implies$ $A+B$ of type $(\mathrm{S})_+$.
\end{itemize}
Throughout the paper, $\L^h$ indicates the Lebesgue $\sigma$-algebra of $\R^h$, $|E|_h$ is the $h$-dimensional Lebesgue measure of $E\in\L^h$, while
$${\rm diam}(E):=\sup\{|x-y|:x,y\in E\}$$
provided $E\subseteq\R^h$, $E\neq\emptyset$. Further,
$$t_\pm:=\max\{\pm t,0\},\quad t\in\R.$$
If $X(\Omega)$ is a real-valued function space on $\Omega$ and $u,v\in X(\Omega)$ then $u\leq v$ means $u(x)\leq v(x)$ a.e.\,in $\Omega$. Analogously for $u<v$, etc. To shorten notation, define 
$$\{u\geq v\}:=\{x\in \Omega :u(x)\geq v(x)\},\quad X(\Omega)_+:=\{w\in X(\Omega): w\geq 0\}.$$
We denote by $p'$ the conjugate exponent of $p$ and $p^*:=\frac{Np}{N-p}$. The Sobolev space $W^{1,p}_0(\Omega)$ is equipped with the Poincaré norm 
\begin{equation*}
\Vert u\Vert_{1,p}:=\Vert |\nabla u|\Vert_p,\quad u\in W^{1,p}_0(\Omega),
\end{equation*}
where, as usual,
\begin{equation*}
\Vert v\Vert_q:=\left\{ 
\begin{array}{ll}
\left(\int_{\Omega }|v(x)|^q\dx\right)^{\frac{1}{q}} & \text{ if }1\leq q<+\infty, \\ 
\phantom{} &  \\ 
\underset{x\in\Omega}{\esssup}\, |v(x)| & \text{ when } q=+\infty.
\end{array}
\right.
\end{equation*}
Let $W^{-1,p'}(\Omega):=W^{1,p}_0(\Omega )^*$. Theorems 9.16 and 6.4 of \cite{Br} ensure that:
\begin{itemize}
\item[$({\rm P}_4)$] $L^{p'}(\Omega) \hookrightarrow W^{-1,p'}(\Omega)$ compactly.
\end{itemize}
Arguing as in \cite[Lemma 2.1]{BG} easily yields:
\begin{itemize}
\item[$({\rm P}_5)$] $u_n\rightharpoonup u$ in $W^{1,p}_0(\Omega)$ $\implies$ $|u_n-u| \rightharpoonup 0$ in $W^{1,p}_0(\Omega)$.
\end{itemize}
Moreover, $v\in L^q(\Omega)\cap W^{-1,p'}(\Omega)$, $q\geq 1$, means that the linear map $u\mapsto \int_\Omega uv\dx$ generated by $v\in L^q(\Omega)$ turns out continuous on $W^{1,p}_0(\Omega)$.

The next three results, where $d:\overline{\Omega}\to[0,+\infty)$ indicates the distance function of $\Omega$, i.e.,
$$d(x):={\rm dist}(x,\partial\Omega)\quad\forall\,x\in\overline{\Omega},$$
will be used in Section \ref{proofmainres}. The second is known as Hardy-Sobolev's inequality.
\begin{prop}[see \cite{GM}, Proposition 2.1]
\label{distsumm}
If $0<\gamma<1<q<\frac{1}{\gamma}$ then $d^{-\gamma}\in L^q(\Omega)$.
\end{prop}
\begin{prop}[cf. \cite{OK}, Theorem 21.3]
\label{hardysobolev}
Let $0<\gamma<1<p<N$. Then there exists $K>0$ such that
$$\int_\Omega d^{-\gamma}|u|\dx\leq K\|u\|_{1,p}\quad\forall\, u\in W^{1,p}_0(\Omega).$$
\end{prop}
\begin{prop}[see \cite{GM}, Proposition 2.3]
\label{distcomparison}
Suppose $u\in C^{1,\alpha}_0(\overline{\Omega })$. Then:
\begin{itemize}
\item[$({\rm a}_1)$] $\big\Vert d^{-1} u\big\Vert_{C^{0,\beta}(\overline{\Omega})}\leq C\Vert u\Vert_{C^{1,\alpha}(\overline{\Omega})}$, where $\beta:=\frac{\alpha}{\alpha+1}$ and $C>0$ does not depend on $u$. 
\item[$({\rm a}_2)$] If $u>0$ in $\Omega$ then there exists $l>0$ such that $u(x)\geq ld(x)$ for every $x\in\Omega$.
\end{itemize}
\end{prop}
Here, as usual,
\begin{equation*}
C^{1,\tau}_0(\overline{\Omega}):=\{u\in C^{1,\tau}(\overline{\Omega}): u\lfloor_{\partial\Omega}=0\},\quad 0<\tau<1.
\end{equation*}
Let $A_p:W^{1,p}_0(\Omega)\to W^{-1,p'}(\Omega)$ be the nonlinear operator stemming from the negative $p$-Laplacian, i.e.,
\begin{equation*}
\langle A_p(u),v\rangle:=\int_\Omega|\nabla u|^{p-2}\nabla u\nabla v\dx\quad\forall\, u,v\in W^{1,p}_0(\Omega).
\end{equation*}
We know \cite{MMP} that
\begin{itemize}
\item[$({\rm P}_6)$] $A_p$ is bounded, continuous, strictly monotone, and of type $({\rm S})_+$\,;
\item[$({\rm P}_7)$] If $\lambda_1$ denotes the first eigenvalue of $(-\Delta_p, W^{1,p}_0(\Omega))$ then there exists a unique eigenfunction $\varphi_1$ associated with $\lambda_1$ and enjoying the properties
\begin{equation}\label{firsteigenf}
\phi_1\in C^{1,\alpha}_0(\overline{\Omega})\;\;\mbox{for some}\;\; 0<\alpha<1,\quad \phi_1>0\;\;\mbox{in}\;\;\Omega,\quad
\Vert\varphi_1\Vert_p=1. 
\end{equation}
\end{itemize} 
\section{Measure-theoretical results}\label{measteorres}
Given any $A\subseteq \R^N$, $x':=(x'_1,\ldots,x'_{N-1})\in\R^{N-1}$, and $i\in\{1,\ldots,N\}$, we introduce the following sets:
\begin{itemize}
    \item {\it projection of $A$ on the hyperplane perpendicular to the coordinate axis $\hat{e}_i$:}
$$\pi_i(A):=\{x'\in\R^{N-1}:(x'_1,\ldots,x'_{i-1},t,x'_i,\ldots,x'_{N-1})\in A\mbox{ for some $t\in\R$}\}.$$
    \item {\it $i$-th section of $A$ at $x'$:}
$$\sigma_i(A,x'):=\{t\in\R: \, (x'_1,\ldots,x'_{i-1},t,x'_i,\ldots,x'_{N-1})\in A\}.$$
  \item{\it projection of $A$ onto the coordinate axis $\hat{e}_i$:}
$$A_i:=\bigcup_{x'\in\pi_i(A)}\sigma_i(A,x').$$
    \item {\it $i$-th fiber of $A$ at $x'$:}
$$ A_i(x'):=\{x\in A:x=(x'_1,\ldots,x'_{i-1},t,x'_i,\ldots,x'_{N-1})\mbox{ for some $t\in\R$}\}. $$
\end{itemize}

\begin{thm}\label{localitythm}
Let $\varphi\in W^{1,1}_\loc(\R^N)$, let $D\subseteq\R^N$ be nonempty, and let $i\in\{1,\ldots,N\}$. Suppose that $|\varphi(D_i(x'))|_1=0$ for $\L^{N-1}$-almost all $x'\in\pi_i(D)$. Then $\partial_i\varphi=0$ $\L^N$-a.e.\,in $D$.
\end{thm}
\begin{proof}
Pick $E\subseteq\R^{N-1}$ satisfying $|E|_{N-1}=0$ and
\begin{equation}\label{neglect1}
|\varphi(D_i(x'))|_1=0 \quad\forall\, x'\in\pi_i(D)\setminus E.
\end{equation}
Given any $x'\in\pi_i(D)\setminus E$, consider the function $\psi_{x'}:\R\to\R$ defined by
$$\psi_{x'}(t):=\varphi(x'_1,\ldots,x'_{i-1},t,x'_i,\ldots,x'_{N-1}), \quad t\in\R. $$
Since $\psi_{x'}(\sigma_i(D,x'))=\varphi(D_i(x'))$, from \eqref{neglect1} it follows
$$|\psi_{x'}(\sigma_i(D,x'))|_1=0. $$
Thus, Theorem 1 of \cite{SV} applied to $\psi_{x'}$ yields $(\psi_{x'})'=0$ $\L^1$-a.e.\,in $\sigma_i(D,x')$. Noticing that $\psi_{x'}\in W^{1,1}_\loc(\R)\hookrightarrow AC_\loc(\R)$ (whence $\psi_{x'}$ is $\L^1$-a.e.\,differentiable) and that 
$$(\psi_{x'})'(t)=\partial_i\varphi(x'_1,\ldots,x'_{i-1},t,x'_i,\ldots,x'_{N-1})\quad
\mbox{for $\L^1$-almost all $t\in\R$,}$$
we can find a set $Z(x')\subseteq\R$ fulfilling
\begin{equation}\label{neglect2}
|Z(x')|_1=0,\quad\partial_i\varphi(x'_1,\ldots,x'_{i-1},t,x'_i,\ldots,x'_{N-1})=0\quad \forall\, t\in \sigma_i(D,x')\setminus Z(x').
\end{equation}
Now, put
$$\Theta:=\{x\in D:\mbox{either }\nexists\,\partial_i\varphi(x) \; \mbox{or } \partial_i\varphi(x)\neq 0\}, $$
where $\nexists\,\partial_i\varphi(x)$ means that $\varphi$ is not differentiable at $x$ along the direction $\hat{e}_i$. Since for every $x\in D$, $x=(x_1,\ldots,x_i,\ldots,x_N)$, one has both $x':=(x_1,\ldots,x_{i-1},x_{i+1},\ldots,x_N)\in\pi_i(D)$ and $x_i\in\sigma_i(D,x')$, setting
$$E':=\{(x_1,\ldots,x_N)\in\R^N:\,(x_1,\ldots,x_{i-1},x_{i+1},x_N)\in E\}$$
and using \eqref{neglect2} we deduce
$$ |\Theta|_N \leq |E'|_N+\int_{\pi_i(D)\setminus E} |Z(x')|_1 \dx'=0$$
because $|E|_{N-1}=0$. Therefore, $\partial_i\varphi=0$ $\L^N$-a.e.\,in $D$, as desired.
\end{proof}
An immediate consequence is the following
\begin{cor}\label{nulldiv}
If $\Phi:=(\varphi_1,\ldots,\varphi_N)\in W^{1,1}_\loc(\R^N;\R^N)$, $D\subseteq\R^N$ is nonempty, and
$$|\varphi_i(D_i(x'))|_1=0 \quad\mbox{for}\;\;\L^{N-1}\mbox{-almost all} \;\; x'\in\pi_i(D),\quad i=1,2,\ldots,N,$$
then $\Div\Phi=0$ $\L^N$-a.e.\,in $D$.
\end{cor}
\begin{ex}\label{example}
Given $\D\subseteq\R^N$ nonempty and $u\in W^{2,1}_\loc(\R^N)$, put $D:=(\nabla u)^{-1}(\D)$, $\Phi:=\nabla u$. 
%Observe that $D_i(x')\subseteq D$ for every $x'\in\pi_i(D)$ and $i\in\{1,\ldots,N\}$.
One clearly has
$$\varphi_i(D)=\partial_i u((\nabla u)^{-1}(\D))=(\nabla u\cdot \hat{e}_i)((\nabla u)^{-1}(\D))=\D_i.$$
Since $D_i(x')\subseteq D$ for every $x'\in\pi_i(D)$, if $|\D_i|_1=0$, $i=1,\ldots,N$, then Corollary \ref{nulldiv} ensures that $\Delta u=\Div\Phi=0$ $\L^N$-a.e.\,in $D$.
\end{ex}
\begin{rmk}\label{degiorgietal}
Taking Example \ref{example} into account, we can state the following.

\textit{Let $\Phi\in W^{1,1}_\loc(\R^N;\R^M)$, $\Phi:=(\varphi_1,\ldots,\varphi_M)$,  and let $D\subseteq\R^N$ be nonempty. If $|\varphi_j(D)|_1=0$ for all $j\in\{1,\ldots,M\}$ then $\Jac\Phi=0$ $\L^N$-a.e.\,in $D$.}

This result is a consequence of Theorem \ref{localitythm}, since $D_i(x')\subseteq D$ for every $x'\in\pi_i(D)$. Anyway, its proof is fairly simpler than the one of Theorem \ref{localitythm}, because it suffices to apply \cite[Lemma 1]{DGBDL} to each $\varphi_j$. On the other hand, Theorem \ref{localitythm} furnishes a detailed proof of \cite[Lemma 1]{DGBDL}.
\end{rmk}
\begin{rmk}
%\label{counterexample}
Given $\Phi\in W^{1,1}_\loc(\R^N;\R^M)$ and $D\subseteq\R^N$ nonempty, one may wonder whether $|\Phi(D)|_M=0$ implies $\Jac\Phi=0$ $\L^N$-a.e.\,in $D$. It is readily seen that this turns out false in the vector case, i.e., $M\geq 2$. In fact, take $D:=\R^N$ and $\Phi\in W^{1,1}_\loc(\R^N;\R^M)$ defined by $\Phi(x_1,\ldots,x_N):=(x_1,0,\ldots,0)$. Then $\Phi(\R^N)=\R\times\{0\}^{M-1}$, which is $\L^M$-negligible, but $\partial_1 \varphi_1=1$ on $\R^N$.
\end{rmk}
\begin{thm}\label{nullproperty}
If $D\subseteq\R^N$ fulfills $|D_i|_1=0$ for every $i\in\{1,\ldots,N\}$ while $\Psi:\R^N\to\R^M$ belongs to $C^{0,1}_\loc(\R^N\setminus\{0\};\R^M)$, then $|\Psi(D)_j|_1=0$ whatever $j=1,\ldots,M$.
\end{thm}
\begin{proof}
Firstly, suppose $D$ bounded and with ${\rm dist}(0,D)>0$. Thus, there exists a compact set  $K\subseteq\R^N\setminus\{0\}$ such that $D\subseteq K$ and ${\rm dist}(D,\R^N\setminus K)>0$. Since $\Psi\lfloor_K\in C^{0,1}(K;\R^M)$, we have
\begin{equation}\label{lipschitzballs}
\diam(\Psi(A))\leq \Lip_K(\Psi)\diam(A)\quad\mbox{for all nonempty}\;\; A\subseteq K,
\end{equation}
where $\Lip_K(\Psi)$ denotes the Lipschitz constant of $\Psi\lfloor_K$. Next, set $\overline{\eps}:=\frac{1}{N}{\rm dist}(D,\R^N\setminus K)$. Fix any $\eps \in(0,\overline{\eps})$ and $i\in\{1,\ldots,N\}$. The condition $|D_i|_1=0$ yields a sequence of open intervals $\{I_n^i\}_n$ satisfying
\begin{equation}\label{covering1}
D_i\subseteq\bigcup_{n=1}^\infty I_n^i \quad \mbox{and} \quad \sum_{n=1}^\infty |I_n^i|_1<\eps.
\end{equation}
Now, notice that
\begin{equation}\label{Dincl}
D\subseteq\prod_{i=1}^N D_i\subseteq\prod_{i=1}^N\bigcup_{n=1}^\infty I_n^i = \bigcup_{h_1,\ldots,h_N=1}^\infty \prod_{i=1}^N I_{h_i}^i    
\end{equation}
and consider the set
$$\Theta:=\left\{Q:=\prod_{i=1}^N I_{h_i}^i: \, (h_1,\ldots,h_N)\in\N^N, \; D\cap Q\neq\emptyset\right\},$$
which turns out countable, because so is $\N^N$. Then $\Theta=\{Q_k\}_k$, where $Q_k$ denotes the $N$-dimensional rectangle
$$ Q_k:=\prod_{i=1}^N I_{h_i}^i \quad\mbox{for some} \;\; (h_1,\ldots,h_N)\in\N^N. $$
From \eqref{Dincl} it evidently follows
$$D\subseteq\
%bigcup_{h_1,\ldots,h_N=1}^\infty \prod_{i=1}^N I_{h_i}^i=
\bigcup_{k=1}^\infty Q_k. $$
Through the triangular inequality we get
\begin{equation}\label{squarediam}
\diam(Q_k)=\diam\left(\prod_{i=1}^N I_{h_i}^i\right)\leq\sum_{i=1}^N \diam(I_{h_i}^i)
=\sum_{i=1}^N |I_{h_i}^i|_1.
\end{equation}
For every $N$-dimensional rectangle $Q_k$ there exists a unique $N$-dimensional ball $B_k\subseteq\R^N$ such that
\begin{equation}
\label{circumcircle}
Q_k\subseteq B_k \quad \mbox{and} \quad \diam(B_k)=\diam(Q_k),
\end{equation}
whence, in particular,
\begin{equation}\label{covering2}
D\subseteq \bigcup_{k=1}^\infty B_k.
\end{equation}
Let $x\in B_k$ and let $y_k\in D\cap Q_k \subseteq D\cap B_k$. Using \eqref{circumcircle}, \eqref{squarediam}, and \eqref{covering1} one arrives at
$${\rm dist}(x,D)\leq |x-y_k|\leq \diam(B_k)=\diam(Q_k)\leq\sum_{i=1}^N |I_{h_i}^i|_1
\leq \sum_{i=1}^N\sum_{n=1}^\infty |I_n^i|_1 < N\eps<N\overline{\eps}.$$
By the choice of $\overline{\eps}$, this entails $x\in\ K$. As $x$ was arbitrary,  $B_k\subseteq K$. If $\tilde{B}_k\subseteq\R^M$ denotes the convex hull of $\Psi(B_k)$ then, thanks to \eqref{lipschitzballs},
\begin{equation}\label{convexhull}
\diam(\tilde{B}_k)=\diam(\Psi(B_k))\leq\Lip_K(\Psi)\diam(B_k).
\end{equation}
From \eqref{covering2} we infer
$$\Psi(D) \subseteq \Psi\left(\bigcup_{k=1}^\infty B_k\right) = \bigcup_{k=1}^\infty \Psi(B_k) \subseteq \bigcup_{k=1}^\infty \tilde{B}_k$$
and, a fortiori,
\begin{equation}\label{covering3}
\Psi(D)_j\subseteq\left(\bigcup_{k=1}^\infty \tilde{B}_k\right)_j=
\bigcup_{k=1}^\infty (\tilde{B_k})_j\quad\forall\, j=1,\ldots,M.
\end{equation}
Each set $\tilde{I}_k^j:=(\tilde{B}_k)_j$ is an interval. Gathering \eqref{convexhull}, \eqref{circumcircle}, \eqref{squarediam}, and \eqref{covering1} together produces
\begin{equation}\label{measureest}
\begin{aligned}
\sum_{k=1}^\infty |\tilde{I}_k^j|_1 & = \sum_{k=1}^\infty \diam(\tilde{I}_k^j) =\sum_{k=1}^\infty\diam(\tilde{B}_k)\leq\Lip_K(\Psi)\sum_{k=1}^\infty\diam(B_k)\\
& = \Lip_K(\Psi)\sum_{k=1}^\infty \diam(Q_k) \leq \Lip_K(\Psi) \sum_{i=1}^N \sum_{n=1}^\infty |I_n^i|_1 < \Lip_K(\Psi)N\eps.
\end{aligned}
\end{equation}
So, by \eqref{covering3}--\eqref{measureest}, for every $\eps\in(0,\overline{\eps})$ there exists a sequence of intervals $\{\tilde{I}_k^j\}_k$ such that
$$\Psi(D)_j\subseteq\bigcup_{k=1}^\infty \tilde{I}_k^j\quad\mbox{and}\quad \sum_{k=1}^\infty |\tilde{I}_k^j|_1<C\eps, $$
where $C:=N\Lip_K(\Psi)$. This forces $|\Psi(D)_j|_1=0$ whatever $j=1,\ldots,M$.

Now, given a generic $D\subseteq\R^N$, consider a sequence $\{K_n\}_n$ of compact sets in $\R^N\setminus\{0\}$ fulfilling $K_n\nearrow \R^N\setminus\{0\}$ and fix $j\in\{1, \ldots,M\}$. Applying the previous argument to $D\cap K_n$ produces $|\Psi(D\cap K_n)_j|_1 = 0$ for all $n\in\N$, whence
\begin{equation*}\begin{aligned}
|\Psi(D)_j|_1 & \leq\left|\left(\Psi\left(\bigcup_{n=1}^\infty (D\cap K_n)
\cup\{0\}\right)\right)_j\right|_1
\leq |\Psi(0)_j|_1+\sum_{n=1}^\infty |\Psi(D\cap K_n)_j|_1=0,
\end{aligned}
\end{equation*}
as desired.
\end{proof}
\begin{ex}\label{example2}
Example \ref{example} can be extended to the $p$-Laplacian. In fact, let $\D\subseteq\R^N$ be nonempty and let $u\in W^{1,1}_\loc(\R^N)$ be such that the field $\Phi(u):=|\nabla u|^{p-2}\nabla u$ belongs to $W^{1,1}_\loc(\R^N;\R^N)$. Define $D:=(\nabla u)^{-1}(\D)$ as well as $\Psi(\xi):=|\xi|^{p-2}\xi$, $\xi\in\R^N$. Evidently, $\Psi\in C^{0,1}_\loc(\R^N\setminus\{0\};\R^N)$ and
$$\varphi_i(D)=(\Psi(\nabla u)\cdot\hat{e}_i)((\nabla u)^{-1}(\D))=\Psi(\D)_i. $$
Thus, if $|\D_i|_1=0$ for every $i=1,\ldots,N$, then Theorem \ref{nullproperty} guarantees that $|\varphi_i(D)|_1=0$ whatever $i$, and Remark \ref{degiorgietal} entails $\Delta_p u=\Div\Phi=0$ $\L^N$-a.e.\,in $D$.
\end{ex}
\begin{ex}
%\label{example3}
At the end of the next section we will apply the argument developed in Example \ref{example2} to the set $\D:=\D_g$. Thus, here we provide a typical example of $g:\R^N\to\R^+$ such that $|(\D_g)_i|_1=0$ for all $i$. Let
$$g(\xi):=\prod_{i=1}^N g_i(\xi_i)\quad\forall\,\xi=(\xi_1,\ldots,\xi_N)\in\R^N,$$
where each $g_i:\R\to[m_i,M_i]$, with $0<m_i<M_i<+\infty$, is a bounded variation function. One evidently has $\D_g\subseteq\prod_{i=1}^N\D_{g_i}$. Moreover, $\D_g$ turns out countable, because so is every $\D_{g_i}$ due to Jordan's decomposition theorem. Consequently, $|(\D_g)_i|_1=0$ for all $i=1,\ldots,N$.\\
Other examples may be obtained by factorization of $g$ into $g_i$'s, being each $g_i$ like in \cite[Examples 3.8-3.9]{GM}.
\end{ex}
\section{Proof of the main result}\label{proofmainres}
As usual \cite{CLM}, a function $\underline{u}\in W^{1,p}(\Omega)$ is called a sub-solution to \eqref{prob} when $\underline{u}\leq 0$ on $\partial\Omega$ and
\begin{equation*}
%\label{defsubsol}
\int_\Omega|\nabla\underline{u}|^{p-2}\nabla\underline{u}\nabla\phi\dx \leq
\int_\Omega \left[f(\underline{u})+g(\nabla\underline{u})\right]\phi\dx\quad\forall\, \phi\in W^{1,p}_0(\Omega)_+\, .
\end{equation*}
\begin{lemma}\label{subsollemma}
Let \eqref{hypatzero} be satisfied. Then problem \eqref{prob} admits a sub-solution $\underline{u}\in C^{1,\alpha}_0(\overline{\Omega})$.
\end{lemma}
\begin{proof}
According to \eqref{hypatzero}, there exists $\delta\in(0,1)$ such that either
\begin{equation}\label{fatzero}
\frac{f(s)}{s^{p-1}}\geq\lambda_1\quad\forall\, s\in(0,2\delta]
\end{equation}
or, with appropriate $\theta>0$,
\begin{equation}\label{gatzero}
g(\xi)\geq\theta\quad\forall\,\xi\in B_{2\delta}(0).
\end{equation}
If \eqref{fatzero} holds, set $\underline{u}:=k\phi_1$, where $\varphi_1$ fulfills \eqref{firsteigenf} while $k>0$ is so small that
\begin{equation}\label{firstuseful}
\Vert\underline{u}\Vert_\infty<\delta.    
\end{equation}
Then
\begin{equation*}
%\label{subsolf}
-\Delta_p \underline{u} = \lambda_1 \underline{u}^{p-1} \leq f(\underline{u}).
\end{equation*}
When \eqref{gatzero} is true, put $\underline{u}:=k\varphi_1$, with $k>0$ small enough to have
\begin{equation}\label{useful}
\|\underline{u}\|_\infty<\left(\frac{\theta}{\lambda_1}\right)^{\frac{1}{p-1}}\quad \mbox{and}\quad\Vert\nabla\underline{u}\Vert_\infty<\delta.
\end{equation}
Consequently,
\begin{equation*}
%\label{subsolg}
-\Delta_p \underline{u} = \lambda_1 \underline{u}^{p-1} \leq \theta \leq g(\nabla\underline{u}).
\end{equation*}
Since $f,g$ are nonnegative, in both cases we infer
$$-\Delta_p\underline{u}\leq f(\underline{u})+g(\nabla\underline{u}),$$
which completes the proof.
\end{proof}
Exploiting the function $\underline{u}$ given by Lemma \ref{subsollemma}, we truncate $f$ as follows:
$$\tilde{f}(x,s):=f(x,\max\{\underline{u}(x),s\}),\quad (x,s)\in\Omega\times\R.$$
Given two approximations of the identity \cite{WZ}, say  $\{\rho_\eps\}_{\eps>0}$ and 
$\{\eta_\eps\}_{\eps>0}$, in $\R$ and in $\R^N$, respectively, let us next regularize both $\tilde{f}$ and $g$ by setting
\begin{equation*}
\tilde{f}_\eps(x,s):=(\tilde{f}(x,\cdot)\ast\rho_\eps)(s)\quad\forall\, (x,s)\in\Omega\times\R,\qquad
g_\eps(\xi):=(g\ast\eta_\eps)(\xi)\quad\forall\, \xi\in\R^N,
\end{equation*}
where $\ast$ denotes the convolution product. Obviously, $\tilde{f}_\eps$ satisfies Carathéodory's conditions while $g_\eps$ is continuous. Put, provided $u\in W^{1,p}_0 (\Omega)$,
\begin{equation}\label{defFG}
F_\eps(u):=\tilde{f}_\eps(\cdot,u)\;\;\mbox{and}\;\; G_\eps(u):=g_\eps(\nabla T(u)),\;\;
\mbox{with}\;\; T(u):=\max\{u,\underline{u}\}.
\end{equation}
\begin{rmk}\label{ordertruncreg}
It should be noted that we truncate and regularize both $f$ and $g$, but in a precise way: $f$ is first truncated and then regularized, whereas the opposite is done for $g$. Interchanging the order seems to be impossible; in fact, the truncation $\tilde{f}$ of $f$ furnishes a good control on the singular term, which remains stable under regularization, while the truncation of $g_\eps$ (the regularization of $g$) must be done via the operator $T$, since it depends on the gradient of $u$, not on $u$ itself. Finally, we highlight that this approach ensures a weak comparison with the sub-solution (see Lemma \ref{comparisonlemma} below).
\end{rmk}
\begin{lemma}\label{fepsest}
Under \ref{fbound}--\ref{fsing}, \eqref{hypatzero}, and \eqref{hypatinfty}, for every $\eps\in(0,1)$ and $\sigma>0$ there exists $C_\sigma>0$ such that
$$\tilde{f}_\eps(x,s)\leq (L_f+\sigma)|s|^{p-1}+C_\sigma d(x)^{-\gamma}\quad\forall\, (x,s)\in\Omega\times\R, $$
where $\displaystyle{L_f:=\limsup_{s\to+\infty}\frac{f(s)}{s^{p-1}}}$.
\end{lemma}
\begin{proof}
Due to \eqref{hypatinfty} one has $L_f<\infty$. Fix $\eps\in (0,1)$ and $\sigma>0$. Assumptions \ref{fbound}--\ref{fsing} yield $c_\sigma >0$ fulfilling
$$f(s)\leq c_\sigma s^{-\gamma}+\left(L_f+\frac{\sigma}{2}\right)s^{p-1}\quad\mbox{in} \;\;\R^+. $$
By the definition of $\tilde{f}$ we thus get, enlarging $c_\sigma$,
$$\tilde{f}(x,s)\leq c_\sigma\underline{u}(x)^{-\gamma}+
\left(L_f+\frac{\sigma}{2}\right)|s|^{p-1},\quad (x,s)\in\Omega\times\R.$$
Therefore, 
\begin{equation*}
\begin{aligned}
\tilde{f}_\eps (x,s) & = \int_{B_\eps(s)}\tilde{f}(x,t)\rho_\eps(s-t)\dt
\leq\|\rho_\eps\|_1 \esssup_{|s-t|<\eps}\tilde{f}(x,t) \\
& \leq c_\sigma \underline{u}(x)^{-\gamma}
+\left(L_f+\frac{\sigma}{2}\right)(|s|+\eps)^{p-1}\leq c_\sigma \underline{u}(x)^{-\gamma} +\left(L_f+\frac{\sigma}{2}\right)(|s|+1)^{p-1},
\end{aligned}
\end{equation*}
because $\supp\rho_\eps \subseteq B_\eps(0)$, $\|\rho_\eps\|_1=1$, and $\eps\in(0,1)$. A simple computation shows that 
$$\left(L_f+\frac{\sigma}{2}\right)(|s|+1)^{p-1}\leq (L_f+\sigma)|s|^{p-1}+c'_\sigma\quad 
\forall\, s\in\R,$$
with appropriate $c_\sigma'>0$. Thus,
\begin{equation}\label{tildefeps}
\tilde{f}_\eps (x,s)\leq c_\sigma \underline{u}(x)^{-\gamma}+(L_f+\sigma)|s|^{p-1}+c'_\sigma\quad\mbox{in}\;\;\Omega\times\R.
\end{equation}
Since $\underline{u}\in C^{1,\alpha}_0(\overline{\Omega})$, from Proposition \ref{distcomparison} it follows $\underline{u}\geq ld$ for some $l>0$. Now, the conclusion is achieved by simply choosing $C_\sigma:= c_\sigma l^{-\gamma}+c'_\sigma \diam(\Omega)^\gamma$.
\end{proof}
\begin{lemma}\label{gepsest}
Let \ref{gbound}, \eqref{hypatzero}, and \eqref{hypatinfty} be satisfied. Then, for every $\eps\in(0,1)$ and $\sigma>0$ there exists $C'_\sigma>0$ such that
$$ g_\eps(\xi)\leq (L_g+\sigma)|\xi|^{p-1}+C'_\sigma\quad\forall\,\xi\in\R^N, $$
where $\displaystyle{L_g:=\limsup_{|\xi|\to+\infty}\frac{g(\xi)}{|\xi|^{p-1}}}$.
\end{lemma}
\begin{proof}
Thanks to \eqref{hypatinfty} one has $L_g<\infty$. Fix $\sigma>0$. Assumption \ref{gbound} and the choice of $L_g$ provide  $c_\sigma>0$ fulfilling
$$g(\xi)\leq c_\sigma+\left(L_g+\frac{\sigma}{2}\right)|\xi|^{p-1}\quad\mbox{in}
\;\;\R^N. $$
Since $\supp\eta_\eps\subseteq B_\eps(0)$, $\|\eta_\eps\|_1=1$, and $\eps\in(0,1)$, we next have
\begin{equation}\label{geps}
\begin{aligned}
g_\eps(\xi)& =\int_{B_\eps(\xi)} g(\zeta)\eta_\eps(\xi-\zeta) \dzeta \leq \|\eta_\eps\|_1 \sup_{\zeta\in B_\eps(\xi)} \left[c_\sigma
+\left(L_g+\frac{\sigma}{2}\right)|\zeta|^{p-1}\right] \\
& \leq c_\sigma + \left(L_g+\frac{\sigma}{2}\right)(|\xi|+1)^{p-1}\leq c_\sigma + (L_g+\sigma)|\xi|^{p-1} + c'_\sigma
\end{aligned}
\end{equation}
for all $\xi\in\R^N$, where the inequality
$$\left(L_g+\frac{\sigma}{2}\right)(|\xi|+1)^{p-1}\leq (L_g+\sigma)|\xi|^{p-1}
+ c'_\sigma,\quad\xi\in\R^N,$$
was used (cf. the previous proof). Setting $C'_\sigma := c_\sigma+c'_\sigma$ leads to the conclusion.
\end{proof}
\begin{thm}\label{exauxprob}
Suppose \ref{fbound}--\ref{fsing}, \ref{gbound}, \eqref{hypatzero}, and \eqref{hypatinfty} hold true. Then, for every $\eps\in(0,1)$, the equation
\begin{equation}\label{auxprob}
-\Delta_p u-F_\eps(u)-G_\eps(u)=0\quad\mbox{in} \;\; W^{-1,p'}(\Omega),
\end{equation}
where $F_\eps$ and $G_\eps$ are defined in \eqref{defFG}, admits a solution $u_\eps\in W^{1,p}_0(\Omega)$.
\end{thm}
\begin{proof}
Henceforth, $({\rm P}_i)$ stands for `property $({\rm P}_i)$ of Section \ref{prel}'. A simple argument based on Lemma \ref{gepsest} and $({\rm P}_4)$ ensures that the map $G_\eps:W^{1,p}_0(\Omega) \to W^{-1,p'}(\Omega)$ is continuous and compact. Concerning $F_\eps$, write
\begin{equation*}
b_\eps(u):=\frac{F_\eps(u)}{|u|^{p-1}+d^{-\gamma}}\, .   
\end{equation*}
Thanks to Lemma \ref{fepsest} one has 
\begin{equation}\label{bepsest}
%b_\eps(u)\in L^\infty(\Omega)\quad\mbox{and}\quad
\Vert b_\eps(u)\Vert_\infty\leq C^*_\sigma\quad\forall\, u\in W^{1,p}_0(\Omega),  
\end{equation}
with appropriate $C^*_\sigma>0$. Moreover,
\begin{equation*}
F_\eps(u)=H_\eps(u)+K_\eps(u),\quad\mbox{where}\quad H_\eps(u):=b_\eps(u)|u|^{p-1},\;\;
K_\eps(u):=b_\eps(u)d^{-\gamma}.
\end{equation*}
Through \eqref{bepsest} and $({\rm P}_4)$ we easily see that $H_\eps:W^{1,p}_0(\Omega)\to W^{-1,p'}(\Omega)$ is a compact continuous operator. By Proposition \ref{hardysobolev}, $K_\eps$ maps $W^{1,p}_0(\Omega)$ into $W^{-1,p'}(\Omega)$. Let us next show that $\pm K_\eps$ turn out pseudo-monotone. If $u_n\rightharpoonup u$ in $W^{1,p}_0(\Omega)$ then, due to $({\rm P}_5)$, $|u_n-u|\rightharpoonup 0$. Since, by \eqref{bepsest} again,
$$\left|\langle K_\eps(u_n),u_n-u\rangle\right|\leq 
C^*_\sigma\int_\Omega d^{-\gamma}|u_n-u|\,{\rm d}x\quad\forall\, n\in\N,$$
from Proposition \ref{hardysobolev} it follows
\begin{equation}\label{Kepsest}
\lim_{n\to+\infty}\langle K_\eps(u_n),u_n-u\rangle=0.   
\end{equation}
On the other hand, along a subsequence when necessary, 
$$u_n\rightharpoonup u\;\;\mbox{in}\;\; W^{1,p}_0(\Omega)\implies u_n\to u\:\: \mbox{a.e.\,in}\;\;\Omega\implies b_\eps(u_n)\to b_\eps(u)\;\;\mbox{a.e.\,in}\;\;\Omega$$
while, for any $v\in W^{1,p}_0(\Omega)$, inequality \eqref{bepsest} entails 
$$|b_\eps(u_n)d^{-\gamma}v|\leq C^*_\sigma d^{-\gamma}|v|,\quad n\in\N.$$
Because of Proposition \ref{hardysobolev}, the dominated convergence theorem applies, and one arrives at
$$\lim_{n\to+\infty}\int_\Omega b_\eps(u_n) d^{-\gamma}v\,{\rm d}x=
\int_\Omega b_\eps(u) d^{-\gamma}v\,{\rm d}x\, .$$
As $v\in W^{1,p}_0(\Omega)$ was arbitrary,
\begin{equation}\label{Kepslim}
\lim_{n\to+\infty} K_\eps(u_n)=K_\eps(u)\;\;\mbox{weakly in}\;\; W^{-1,p'}(\Omega),    
\end{equation}
Now, gathering \eqref{Kepsest}--\eqref{Kepslim} together yields
$$\langle K_\eps(u_n),u_n-v\rangle=
\langle K_\eps(u_n),u_n-u\rangle+\langle K_\eps(u_n),u-v\rangle\to 0$$
as $n\to+\infty$, whence both $K_\eps$ and $-K_\eps$ are pseudo-monotone. 
Via $({\rm P}_6)$, the compactness of $G_\eps$ and $H_\eps$, besides $({\rm P}_2)$--$({\rm P}_3)$, we can verify that $A_p-(H_\eps+G_\eps)$ is pseudo-monotone. So, the map
$$\Psi_\eps(u):=-\Delta_p u-(F_\eps+G_\eps)(u)
=A_p(u)-(H_\eps+G_\eps)(u)-K_\eps(u),\quad u\in W^{1,p}_0(\Omega),$$
also turns out to be pseudo-monotone; see $({\rm P}_1)$. By Theorem 2.99 in \cite{CLM}, the conclusion is achieved once we show that $\Psi_\eps$ is coercive. Exploiting Lemmas \ref{fepsest}--\ref{gepsest}, Proposition \ref{hardysobolev}, H\"older's inequality, and Poincaré's inequality, provides
\begin{equation*}
\begin{aligned}
\langle & \Psi_\eps(u),u \rangle \\
& \geq\|\nabla u\|_p^p-\int_\Omega \left[(L_f+\sigma)|u|^{p-1}
+C_\sigma d^{-\gamma}\right]|u|\dx
-\int_\Omega\left[(L_g+\sigma)|\nabla T(u)|^{p-1}+C'_\sigma\right]|u|\dx \\
& \geq\|\nabla u\|_p^p-(L_f+\sigma)\|u\|_p^p-C_\sigma''\|\nabla u\|_p
-(L_g+\sigma)\left(\|\nabla u\|_p^{p-1}+\|\nabla\underline{u}\|_p^{p-1}\right)\|u\|_p-C'_\sigma\|u\|_1 \\
& \geq\left(1-\frac{L_f+\sigma}{\lambda_1}-\frac{L_g+\sigma}{\lambda_1^{1/p}}\right) \|\nabla u\|_p^p - \tilde{C}_\sigma\|\nabla u\|_p\, ,
\end{aligned}
\end{equation*}
namely
\begin{equation}\label{coercive}
\langle\Psi_\eps(u),u\rangle\geq\left(1-\frac{L_f+\sigma}{\lambda_1}
-\frac{L_g+\sigma}{\lambda_1^{1/p}}\right)\|\nabla u\|_p^p-\tilde{C}_\sigma\|\nabla u\|_p 
\quad\forall\, u\in W^{1,p}_0(\Omega),
\end{equation}
where $\tilde{C}_\sigma$ denotes a suitable positive constant (that also depends on $\|\nabla\underline{u}\|_p$). Due to \eqref{hypatinfty}, we can choose $\sigma>0$ fulfilling
\begin{equation}\label{sigmacond}
1-\frac{L_f+\sigma}{\lambda_1}-\frac{L_g+\sigma}{\lambda_1^{1/p}}>0.
\end{equation}
Accordingly, $\frac{\langle\Psi_\eps(u),u\rangle}{\|\nabla u\|_p}\to+\infty$ as $\|\nabla u\|_p\to\infty$, which completes the proof.
\end{proof}
\begin{lemma}\label{comparisonlemma}
Let \ref{fbound}--\ref{fsing}, \ref{gbound}, \eqref{hypatzero}, and \eqref{hypatinfty} be satisfied. If $0<\eps<\delta$, where $\delta$ complies with \eqref{fatzero}--\eqref{gatzero}, then for every solution $u_\eps\in W^{1,p}_0(\Omega)$  to \eqref{auxprob} one has $u_\eps\geq \underline{u}$. 
\end{lemma}
\begin{proof}
By \eqref{defFG}, testing \eqref{auxprob} with $(\underline{u}-u_\eps)_+\in W^{1,p}_0(\Omega)$ yields
\begin{equation}
\label{comparison1}
\int_{\{\underline{u}\geq u_\eps\}}|\nabla u_\eps|^{p-2}\nabla u_\eps\cdot\nabla (\underline{u}-u_\eps)\dx
=\int_{\{\underline{u}\geq u_\eps\}}(\tilde{f}_\eps(\cdot,u_\eps)+g_\eps(\nabla\underline{u}))(\underline{u}-u_\eps) \dx.
\end{equation}
Two cases can now arise, depending on whether \eqref{fatzero} or \eqref{gatzero} is true. When \eqref{fatzero} holds, $\Vert\underline{u}\Vert_\infty<\delta$ by \eqref{firstuseful} and, through an argument similar to that made just above \eqref{tildefeps}, we obtain
\begin{equation}\label{comparisonf}
\begin{aligned}
&\int_{\{\underline{u}\geq u_\eps\}}(\tilde{f}_\eps(\cdot,u_\eps)+g_\eps(\nabla\underline{u}))(\underline{u}-u_\eps)\dx
\geq\int_{\{\underline{u}\geq u_\eps\}}
\tilde{f}_\eps(\cdot,u_\eps)(\underline{u}-u_\eps)\dx\\
&\geq\int_{\{\underline{u}\geq u_\eps\}}\left[\essinf_{|s-u_\eps|<\eps} \tilde{f}(\cdot,s)\right](\underline{u}-u_\eps)\dx
\geq\int_{\{\underline{u}\geq u_\eps\}} \left[\essinf_{s\in[\underline{u},2\delta)} f(s)\right](\underline{u}-u_\eps)\dx\\
&\geq\lambda_1\int_{\{\underline{u}\geq u_\eps\}}\underline{u}^{p-1}(\underline{u}-u_\eps)\dx,
\end{aligned}
\end{equation}
because $u_\eps(x)+\eps\leq\|\underline{u}\|_\infty+\delta<2\delta$ for almost every 
$x\in\{\underline{u}\geq u_\eps\}$. On the other hand, if \eqref{gatzero} is fulfilled then \eqref{useful} implies $\max\{\Vert\nabla\underline{u}\Vert_\infty,\eps\}<\delta$. Thus,
\begin{equation}\label{comparisong}
\begin{aligned}
&\int_{\{\underline{u}\geq u_\eps\}}(\tilde{f}_\eps(\cdot,u_\eps)+g_\eps(\nabla\underline{u}))(\underline{u}-u_\eps)\dx
\geq\int_{\{\underline{u}\geq u_\eps\}} g_\eps(\nabla\underline{u})(\underline{u}-u_\eps) \dx\\
& \geq\int_{\{\underline{u}\geq u_\eps\}}\left[\essinf_{\xi\in B_\eps(\nabla\underline{u})} g(\xi)\right](\underline{u}-u_\eps)\dx 
\geq\int_{\{\underline{u}\geq u_\eps\}}\left[\essinf_{\xi\in B_{2\delta}(0)} g(\xi)\right](\underline{u}-u_\eps)\dx\\
&\geq\int_{\{\underline{u}\geq u_\eps\}}\theta(\underline{u}-u_\eps)\dx\geq \lambda_1\int_{\{\underline{u}\geq u_\eps\}}\underline{u}^{p-1}(\underline{u}-u_\eps)\dx.
\end{aligned}
\end{equation}
In both cases, from \eqref{comparison1} and either \eqref{comparisonf} or \eqref{comparisong} it follows
\begin{equation*}
%\label{comparison2}
\begin{aligned}
\int_{\{\underline{u}\geq u_\eps\}} |\nabla u_\eps|^{p-2}\nabla u_\eps\cdot \nabla(\underline{u}-u_\eps)\dx 
& \geq\lambda_1\int_{\{\underline{u}\geq u_\eps\}}
\underline{u}^{p-1}(\underline{u}-u_\eps)\dx \\
& =\int_{\{\underline{u}\geq u_\eps\}}|\nabla \underline{u}|^{p-2}\nabla\underline{u}\cdot \nabla(\underline{u}-u_\eps)\dx,
\end{aligned}
\end{equation*}
whence
$$\int_{\{\underline{u}\geq u_\eps\}}(|\nabla \underline{u}|^{p-2}\nabla\underline{u}-|\nabla u_\eps|^{p-2}\nabla u_\eps)\cdot(\nabla\underline{u}-\nabla u_\eps)\dx\leq 0.$$
Due to the strict monotonicity of the $p$-Laplacian, this easily entails $u_\eps\geq \underline{u}$. 
\end{proof}
\begin{lemma}\label{reglemma}
Under the assumptions of Lemma \ref{comparisonlemma} and the same choice of $\delta$, any solution $u_\eps\in W^{1,p}_0(\Omega)$, with $0<\eps<\delta$, to \eqref{auxprob} lies in $C^{1,\beta}_0(\overline{\Omega})$. Moreover, the $C^{1,\beta}$-estimates do not depend on $\eps$.
%that is, there exists $K>0$, independent of $\eps\in(0,\delta)$, such that, for any $\eps\in(0,\delta)$, every $u_\eps$ solution to \eqref{auxprob} satisfies 
%$$ \|u_\eps\|_{C^{1,\beta}} \leq K. $$
\end{lemma}
\begin{proof}
Let $\sigma>0$ satisfy \eqref{sigmacond} and let $u_\eps$ be as above. Lemma \ref{comparisonlemma} yields $u_\eps\geq\underline{u}$. Using \eqref{coercive} we next get
$$0=\langle\Psi_\eps(u_\eps),u_\eps\rangle
\geq\left(1-\frac{L_f+\sigma}{\lambda_1}-\frac{L_g+\sigma}{\lambda_1^{1/p}}\right)
\|\nabla u_\eps\|_p^p-\tilde{C}_\sigma\|\nabla u_\eps\|_p\, .$$
So, there exists $\hat{C}>0$, independent of $\eps$, such that
\begin{equation}\label{energyest}
\|\nabla u_\eps\|_p \leq \hat{C}.
\end{equation}
Pick $k\geq \hat{k}:=1+\|\underline{u}\|_\infty$ and set $\Omega_k:=\{u_\eps\geq k\}$. We claim that for appropriate $A,B,C>0$, which do not depend on either $\eps$ or $k$, one has
\begin{equation}\label{generalgrowth}
\tilde{f}_\eps(x,u_\eps(x))+g_\eps(\nabla u_\eps(x))
\leq A|\nabla u_\eps(x)|^{p-1}+Bu_\eps(x)^{p-1}+C\quad\mbox{a.e.\,in}\;\;\Omega_k.
\end{equation}
In fact, via the same reasoning employed to prove \eqref{tildefeps}--\eqref{geps}, from 
$$u_\eps(x)-\eps\geq k-\eps>k-1\geq \|\underline{u}\|_\infty\quad\mbox{for almost all}\;\; x\in\Omega_k$$
it follows
\begin{equation*}
\begin{aligned}
\tilde{f}_\eps(\cdot,u_\eps)+g_\eps(\nabla u_\eps) 
& \leq\esssup_{|s-u_\eps|<\eps}\tilde{f}(\cdot,s)+\esssup_{|\xi-\nabla u_\eps|<\eps} g(\xi)=\esssup_{|s-u_\eps|<\eps} f(s)+\esssup_{|\xi-\nabla u_\eps|<\eps} g(\xi)\\
&\leq (L_f+\sigma)u_\eps^{p-1} + (L_g+\sigma)|\nabla u_\eps|^{p-1}+(u_\eps-\eps)^{-\gamma} + c_\sigma \\
&\leq (L_f+\sigma)u_\eps^{p-1}+(L_g+\sigma)|\nabla u_\eps|^{p-1}+ \|\underline{u}\|_\infty^{-\gamma}+c_\sigma
\end{aligned}
\end{equation*}
a.e.\,in $\Omega_k$, and \eqref{generalgrowth} holds. After testing \eqref{auxprob} with $(u_\eps-k)_+\in W^{1,p}_0(\Omega)$, inequality \eqref{generalgrowth} and Young's inequality with $\mu>0$ produce
\begin{equation}\label{degiorgiestabove}
\begin{aligned}
\int_{\Omega_k}|\nabla u_\eps|^p \dx 
& = \int_\Omega |\nabla u_\eps|^{p-2}\nabla u_\eps\cdot\nabla(u_\eps-k)_+\dx\\
& = \int_\Omega\left[\tilde{f}_\eps(\cdot,u_\eps)+g_\eps(\nabla u_\eps)\right](u_\eps-k)_+\dx\\
& \leq\int_{\Omega_k} \left(A|\nabla u_\eps|^{p-1}+Bu_\eps^{p-1}+C\right)u_\eps\dx\\
& \leq\mu A\int_{\Omega_k}|\nabla u_\eps|^p\dx+(C_\mu A+B)\int_{\Omega_k} u_\eps^p\dx
+C\int_{\Omega_k} u_\eps\dx\\
& \leq\mu A\int_{\Omega_k}|\nabla u_\eps|^p\dx +(C_\mu A+B+C)\int_{\Omega_k} u_\eps^{p^*}\dx\\
& \leq\mu A\int_{\Omega_k}|\nabla u_\eps|^p\dx+C_\mu^*(\|(u_\eps-k)_+\|_{p^*}^{p^*}
+ k^{p^*}|\Omega_k|)
\end{aligned}
\end{equation}
(note that $u_\eps\geq k>1$ in $\Omega_k$), where $C_\mu^*:=2^{p^*-1}(C_\mu A+B+C)$. By Sobolev's inequality we obtain 
\begin{equation}\label{degiorgiestbelow}
c\|(u_\eps-k)_+\|_{p^*}^p\leq\frac{1}{2}\int_{\Omega_k}|\nabla u_\eps|^p\dx
\end{equation}
for suitable $c>0$. Now, choose $\mu<\frac{1}{2A}$ and use \eqref{degiorgiestabove}--\eqref{degiorgiestbelow} to achieve
$$ c\|(u_\eps-k)_+\|_{p^*}^p \leq (1-\mu A)\int_{\Omega_k}|\nabla u_\eps|^p \dx \leq C_\mu^* (\|(u_\eps-k)_+\|_{p^*}^{p^*} + k^{p^*}|\Omega_k|). $$
Since $k\geq\hat{k}$ was arbitrary, Lemma 3.2 of \cite{CGL} gives a constant $M>0$ fulfilling
\begin{equation}\label{supest}
\|u_\eps\|_\infty \leq M.
\end{equation}
It should be noted that $M$ does not depend on $\eps$, because so is $\hat{C}$ in \eqref{energyest}. Thanks to \eqref{supest}, from Lemmas \ref{fepsest}--\ref{gepsest} it follows
\begin{equation}\label{generalgrowth2}
0\leq\tilde{f}_\eps(\cdot,u_\eps)+g_\eps(\nabla u_\eps)
\leq\tilde{C}\left(d^{-\gamma}+|\nabla u_\eps|^{p-1}\right),
\end{equation}
with appropriate $\tilde{C}>0$. Thus, the function
$$b_\eps(x):=\frac{\tilde{f}_\eps(x,u_\eps(x))+g_\eps(\nabla u_\eps(x))}{d(x)^{-\gamma} +|\nabla u_\eps(x)|^{p-1}}\, ,\quad x\in\Omega,$$
belongs to $L^\infty(\Omega)$ and $\|b_\eps\|_\infty\leq \tilde{C}$, while $u_\eps$ satisfies
\begin{equation}\label{hahnbanach}
\left\{
\begin{alignedat}{2}
-\Delta_p u_\eps & =b_\eps(x)\left(d(x)^{-\gamma}+|\nabla u_\eps(x)|^{p-1}\right)\quad && \mbox{in}\;\;\Omega,\\
u_\eps & =0\quad && \mbox{on} \;\; \partial\Omega.
\end{alignedat}
\right.
\end{equation}
Let $U_\eps\in W^{1,p}_0(\Omega)$ be the unique solution to the problem
\begin{equation*}
\left\{
\begin{alignedat}{2}
-\Delta U & =b_\eps(x)d(x)^{-\gamma}\quad && \mbox{in}\;\;\Omega,\\
U & =0\quad && \mbox{on}\;\;\Omega,
\end{alignedat}
\right.
\end{equation*}
which is directly provided by Minty-Browder's theorem \cite{F}, because Proposition \ref{hardysobolev} ensures that $b_\eps d^{-\gamma}\in W^{-1,p'}(\Omega)$. Due to \cite[Lemma 3.1]{H} one has $U_\eps\in C^{1,\alpha}_0(\overline{\Omega})$ and
\begin{equation}\label{C1est}
\Vert U_\eps\Vert_{C^{1,\alpha}(\overline{\Omega})}\leq\tilde{M},    
\end{equation}
where $\tilde{M}$ does not depend on $\eps$, as $\|b_\eps\|_\infty\leq \tilde{C}$ for any $\eps<\delta$. Moreover, through \eqref{hahnbanach} we see that $u_\eps$ solves the problem
\begin{equation*}
\left\{
\begin{alignedat}{2}
-\Div(|\nabla u|^{p-2}\nabla u-\nabla U_\eps(x)) & =b_\eps(x)|\nabla u|^{p-1}\quad &&\mbox{in}\;\;\Omega,\\
u & =0\quad &&\mbox{on}\;\;\partial\Omega.
\end{alignedat}
\right.
\end{equation*}
At this point,  setting
$$A(x,s,\xi):=|\xi|^{p-2}\xi-\nabla U_\eps(x),\quad B(x,s,\xi):=b_\eps(x)|\xi|^{p-1},\quad (x,s,\xi)\in\Omega\times\R\times\R^N,$$ 
and recalling that $\nabla U_\eps\in C^{0,\alpha}(\overline{\Omega})$, Theorem 1 of \cite{L} can be applied, which yields $u_\eps\in C^{1,\beta}_0(\overline{\Omega})$. Further, since both \eqref{supest} and \eqref{C1est} are uniform with respect to $\eps$, there exists $K>0$, independent of $\eps$, such that $\|u_\eps\|_{C^{1,\beta}(\overline{\Omega})}\leq K$ whatever $\eps\in(0,\delta)$.
\end{proof}
Choose $\eps:=\frac{1}{n}$, where $n\in\N$ fulfills $\frac{1}{n}<\delta$ and $\delta$ comes from \eqref{fatzero}--\eqref{gatzero}. To avoid cumbersome formulae, write $u_n$, $\tilde{f}_n$, $g_n$ for $u_{1/n}$, $\tilde{f}_{1/n}$, $g_{1/n}$ respectively. Lemma \ref{reglemma} and Ascoli-Arzelà's theorem give, along a sub-sequence when necessary, 
\begin{equation}\label{C1conv}
 u_n\to u\;\;\mbox{in}\;\; C^{1,\alpha}_0(\overline{\Omega})\;\;\mbox{for some}\;\;\alpha\in(0,\beta).   
\end{equation}
\begin{lemma}
%\label{weakconvlemma}
Suppose \ref{fbound}--\ref{fsing}, \ref{gbound}, \eqref{hypatzero}, and \eqref{hypatinfty} hold true. Then there exist $v,w\in L^q(\Omega)\cap W^{-1,p'}(\Omega)$, $1<q<\frac{1}{\gamma}$, with the following properties:
\begin{itemize}
\item The function $u$ given by \eqref{C1conv} solves the problem
\begin{equation}\label{weakprob}
\left\{
\begin{alignedat}{2}
-\Delta_p u & =v+w\quad && \mbox{in}\;\;\Omega, \\
u & =0\quad && \mbox{on}\;\;\Omega.
\end{alignedat}
\right.
\end{equation}
\item One has
\begin{equation}\label{vwdoublebound}
\underline{f}(u)\leq v\leq\overline{f}(u),\qquad\underline{g}(\nabla u)\leq w\leq \overline{g}(\nabla u)\, .
\end{equation}
\end{itemize}
\end{lemma}
\begin{proof}
Because of \eqref{C1conv} the sequence $\{\|\nabla u_n\|_\infty\}_n$ turns out bounded. Thus, after possibly enlarging $\tilde{C}$, \eqref{generalgrowth2} entails
\begin{equation}\label{generalgrowth3}
0\leq\min\{\tilde{f}_n(\cdot,u_n),g_n(\nabla u_n)\}
\leq\max\{\tilde{f}_n(\cdot,u_n),g_n(\nabla u_n)\}
\leq\tilde{f}_n(\cdot,u_n)+g_n(\nabla u_n)
\leq\tilde{C} d^{-\gamma}
\end{equation}
for all $n\in\N$, i.e., both $\{\tilde{f}_n(\cdot,u_n)\}_n$ and $\{g_n(\nabla u_n)\}_n$ are bounded in $L^q(\Omega)$ provided $q\in (1,\frac{1}{\gamma})$; see Proposition \ref{distsumm}. Thanks to Proposition \ref{hardysobolev}, \eqref{generalgrowth3} also ensures the uniform boundedness of the linear functionals
$$\phi\mapsto
\int_\Omega\tilde{f}_n(\cdot,u_n)\phi\,{\rm d}x,\quad
\phi\mapsto
\int_\Omega g_n(\nabla u_n)\phi\,{\rm d}x,\quad n\in\N,$$
on $W^{1,p}_0(\Omega)$. Consequently, by Kakutani's and Banach-Alaoglu-Bourbaki's theorems \cite{Br}, there exist $v,w\in L^q(\Omega)\cap W^{-1,p'}(\Omega)$ such that, up to sub-sequences,
\begin{equation}\label{weakconvs}
\begin{aligned}
& \tilde{f}_n(\cdot,u_n)\rightharpoonup v\, ,\quad 
g_n(\nabla u_n)\rightharpoonup w\quad\mbox{in} \;\; L^q(\Omega),\\
& \tilde{f}_n(\cdot,u_n)\stackrel{*}{\rightharpoonup} v\, ,\quad 
g_n(\nabla u_n)\stackrel{*}{\rightharpoonup} w\quad\mbox{in}\;\; W^{-1,p'}(\Omega).
\end{aligned}
\end{equation}
From \eqref{C1conv} it follows $u_n\to u$ in $W^{1,p}_0(\Omega)$, which produces $\Delta_p u_n \to \Delta_p u$ in $W^{-1,p'}(\Omega)$; cf. $({\rm P}_6)$ of Section \ref{prel}. Since Theorem \ref{exauxprob} gives
$$\langle-\Delta_p u_n,\phi\rangle=\langle\tilde{f}_n(x,u_n),\phi\rangle+ 
\langle g_n(\nabla u_n),\phi \rangle,\quad\phi\in W^{1,p}_0(\Omega),$$
for every $n\in\N$, letting $n\to+\infty$ and using \eqref{weakconvs} we see that $u$ solves \eqref{weakprob}. So, it remains to verify \eqref{vwdoublebound}. With this aim, pick any $\mu>0$. By \eqref{C1conv}, there exists $\nu\in\N$ such that
\begin{equation}\label{unifconv}
\frac{1}{n}<\frac{\mu}{2},\quad\|u_n-u\|_\infty<\frac{\mu}{2},\quad\mbox{and}\quad \|\nabla u_n-\nabla u\|_\infty < \frac{\mu}{2}\quad\forall\, n>\nu.
\end{equation}
Proceeding as already done to prove \eqref{tildefeps}, one has
\begin{equation*}
\essinf_{|s-u|<\mu}\tilde{f}(\cdot,s)\leq\essinf_{|s-u_n|<\frac{\mu}{2}}\tilde{f}(\cdot,s)
\leq\tilde{f}_n(\cdot,u_n)\leq\esssup_{|s-u_n|<\frac{\mu}{2}}\tilde{f}(\cdot,s)
\leq \esssup_{|s-u|<\mu}\tilde{f}(\cdot,s)
\end{equation*}
for all $n>\nu$. Hence, 
\begin{equation*}
\int_\Omega\left[\essinf_{|s-u|<\mu}\tilde{f}(\cdot,s)\right]\varphi\dx
\leq\int_\Omega\tilde{f}_n(\cdot,u_n)\,\varphi\dx
\leq\int_\Omega\left[\esssup_{|s-u|<\mu}\tilde{f}(\cdot,s)\right]\varphi\dx
\end{equation*}
whatever $\varphi\in C^\infty_c(\Omega)_+$ and, in view of \eqref{weakconvs},
\begin{equation*}
\int_\Omega\left[\essinf_{|s-u|<\mu}\tilde{f}(\cdot,s)\right]\varphi\dx
\leq\int_\Omega v\,\varphi\dx
\leq\int_\Omega\left[\esssup_{|s-u|<\mu}\tilde{f}(\cdot,s)\right]\varphi\dx.
\end{equation*}
This clearly entails
\begin{equation}\label{vdoublebound}
\essinf_{|s-u|<\mu}\tilde{f}(\cdot,s)\leq v\leq\esssup_{|s-u|<\mu}\tilde{f}(\cdot,s).
\end{equation}
Next, notice that
\begin{equation}\label{neglecttrunc}
\tilde{f}(x,]-s,s[)\subseteq f([\underline{u}(x),s[)\quad\forall\, s>\underline{u}(x).
\end{equation}
Since, by Lemma \ref{subsollemma},
$$u(x)=\lim_{n\to+\infty} u_n(x)\geq\underline{u}(x),$$
\eqref{neglecttrunc} is valid with $s=u(x)+\mu$. Now, gathering \eqref{vdoublebound}--\eqref{neglecttrunc} together and letting $\mu\to 0^+$ we achieve
$$\underline{f}(u)\leq v\leq\overline{f}(u).$$
Likewise, thanks to \eqref{unifconv} and \eqref{weakconvs} again, one easily sees that
\begin{equation*}
%\label{wdoublebound}
\essinf_{|\xi-\nabla u|<\mu} g(\xi)\leq w\leq\esssup_{|\xi-\nabla u|<\mu} g(\xi),
\end{equation*}
whence the conclusion follows as $\mu\to 0^+$.
\end{proof}
\begin{thm}
Let $({\rm H}_f)$, $({\rm H}_g)$, \eqref{hypatzero}, and \eqref{hypatinfty} be satisfied. Then the function $u\in C^{1,\alpha}_0(\overline{\Omega})$ coming from \eqref{C1conv} is a strong solution to \eqref{prob}.
\end{thm}
\begin{proof}
As already done before for \eqref{generalgrowth3}, we get $\overline{f}(u)+\overline{g}(\nabla u)\leq\tilde{C} d^{-\gamma}$ provided $\tilde{C}$ is big enough. Consequently, from
$$0\leq\min\{\underline{f}(u),\underline{g}(\nabla u)\}\leq
\max\{\overline{f}(u),\overline{g}(\nabla u)\}\leq
\overline{f}(u)+\overline{g}(\nabla u)\leq\tilde{C} d^{-\gamma},$$
$d^{-\gamma}\in L^\infty_\loc(\Omega)$, and \eqref{vwdoublebound} it follows $v,w\in L^\infty_\loc(\Omega)$. Recalling that $u$ weakly solves \eqref{weakprob},  Theorem 2.1 of \cite{CM} yields $|\nabla u|^{p-2}\nabla u\in W^{1,2}_\loc(\Omega)$, namely $u$ is actually a strong solution. It remains to verify that $v=f(u)$ and $w=g(\nabla u)$.

Let $\Omega_f:=u^{-1}(\D_f)$. Via \ref{fdisc} and \cite[Proposition 2.4]{GM} we arrive at $\Delta_p u=0$ a.e.\,in $\Omega_f$. By \eqref{weakprob}--\eqref{vwdoublebound} this entails
$$0\leq\underline{f}(u(x))\leq v(x)\leq v(x)+w(x)=-\Delta_p u(x)=0 \quad\mbox{for almost all}\;\; x\in\Omega_f.$$
Thus, \ref{fzeros} implies $f(u)=0=v$ a.e.\,in $\Omega_f$. On the other hand, thanks to \eqref{vwdoublebound}, one has $v(x)=f(u(x))$ for almost all $x\in\Omega\setminus\Omega_f$, whence $v=f(u)$.

Put $\Omega_g:=(\nabla u)^{-1}(\D_g)$. Through \ref{gdisc} and Theorem \ref{nullproperty} (see also Example \ref{example2}) we infer
$$|(\Phi(u)(\Omega_g))_i|_1=0,\quad i=1,\ldots,N,$$
where $\Phi(u):=|\nabla u|^{p-2}\nabla u$. From Corollary \ref{nulldiv} (cf. Example \ref{example2} again) it thus follows $\Delta_p u = \Div(\Phi(u))=0$ a.e.\,in $\Omega_g$.
Because of \eqref{weakprob}--\eqref{vwdoublebound} this produces
$$0\leq\underline{g}(\nabla u(x))\leq w(x)\leq v(x)+w(x)=-\Delta_p u(x)=0 \quad\mbox{for almost all} \;\; x\in\Omega_g,$$
which actually means $g(\nabla u)=0=w$ a.e.\,in $\Omega_g$ due to \ref{gzeros}. On the other hand, \eqref{vwdoublebound} guarantees that $w(x)=g(\nabla u(x))$ for almost all $x\in\Omega\setminus\Omega_g$. So, $w=g(\nabla u)$.
\end{proof}
\section*{Acknowledgments}
\noindent
%This study was partly funded by: Research project of MIUR (Italian Ministry of Education, University and Research) Prin 2022 {\it Nonlinear differential problems with applications to real phenomena} (Grant No. 2022ZXZTN2).
The authors are members of the {\em Gruppo Nazionale per l'Analisi Matematica, la Probabilit\`a e le loro Applicazioni} (GNAMPA) of the {\em Istituto Nazionale di Alta Matematica} (INdAM). \\
The first author acknowledges the support of the INdAM-GNAMPA project ``Classificazione delle soluzioni di equazioni evolutive non locali'' (CUP E53C25002010001).

\end{document}